\documentclass[12pt,a4paper,reqno]{amsart}%{article}%The reqno document class option displays equation numbers on the right
\usepackage[utf8]{inputenc}
\usepackage{enumitem}
\usepackage{amsmath}
\usepackage{amsfonts}
\usepackage{amssymb}
\usepackage[all]{xy}
\usepackage{xfrac}
\usepackage[yyyymmdd,hhmmss]{datetime}

\usepackage[english]{babel}
\usepackage{dsfont}

\newtheorem{theorem}{Theorem}[section]
\newtheorem{corollary}[theorem]{Corollary}

\newtheorem{definition}{Definition}[section]
\newtheorem{proposition}[theorem]{Proposition}

\newcommand{\two}{\textnormal{\emph{2}}}
\newcommand{\oalph}{\beta}
\newcommand{\muu}{\text{\sfrac{1}{2}}}%{\raisebox{-0.4ex}{\text{\sfrac{1}{2}}}}
 \author{J. P. Fatelo and N. Martins-Ferreira}
 \address{IPL}
 \email{martins.ferreira@ipleiria.pt}
\title[]{Mobi algebra as an abstraction to the unit interval and its comparison to rings}


\thanks{This research work was supported by the Portuguese Foundation for Science and Technology (FCT) through the 
Project reference UID/ Multi/04044/2013,
post doctoral grant SFRH/BPD/43216/2008 at CMUC,
and also by CDRSP and ESTG from the Polytechnic Institute of Leiria.
}

\begin{document}

\begin{abstract}

We begin by introducing an algebraic structure with three constants and one ternary operation to which we call mobi algebra. This structure has been designed  to capture the most relevant properties of the unit interval that are needed in the study of geodesic paths. Another algebraic structure, called involutive medial monoid (IMM), can be  derived from a mobi algebra. We prove several results on the interplay between mobi algebras, IMM algebras and unitary rings. It turns out that every unitary ring with one half uniquely determines and is uniquely determined by a mobi algebra with one double. This paper is the second of a planned series of papers dedicated to the study of geodesic paths from an algebraic point of view, the first paper in the series is \cite{ccm_magmas}.

\end{abstract}

\maketitle

%\today; \currenttime

\section{Introduction}

This paper is part of an ongoing work which aims at an axiomatic  characterization of spaces with unique geodesic paths between any two of its points.
Our interest in the study of geodesic paths from an algebraic point of view has started with an observation on weakly Mal'cev categories \cite{NMF.08}.
% The class consists on those algebraic structures of the form $(A,p)$ with $A$ a set and $p(x,y,z)$ a ternary operation; moreover, the axiom $p(x,y,y)=p(y,y,x)$ holds true for all $x,y\in A$, as well as the assertion that if for some $a\in A$, $p(x,a,a)=p(y,a,a)$, then $x=y$. Any Mal'tsev variety in the sense of universal algebra is an example (see \cite{NMF.15a} for further details). 
 The observation is that a certain ternary operation, say $p=p(x,y,z)$, which arises in the context of weakly Mal'tsev quasi-varieties \cite{NMF.15a}, could be used to describe the path of a particle moving in space along a geodesic curve from a point $x$ to a point $z$ at an instant $y$. 
By fixing $y$ to a value that positions the particle at half way from $x$ to $z$, a binary operation is obtained. The study of this binary operation was a first step in our investigation \cite{ccm_magmas}.

Allowing the variable $y$ to range over the unit interval we are able to capture the whole movement of a particle in a geodesic path. The main purpose of this paper is to communicate the structure of a mobi algebra (Definition \ref{mobi_algebra}), which has arisen as a candidate for an algebraic version of the unit interval. 
It may be argued that this study could have been avoided if we would have decided to simply work within the unit interval. The reason why we have chosen to take this study further has to do with the fact that the axioms of Definition \ref{mobi_algebra} can be naturally extended into higher dimensions. This means that a mobi algebra is at the same time a model for the unit interval and an instance of a one-dimensional space with geodesics.
However, this is just an intermediate step in our study. 
The notion which we are aiming for in the future is the one of a mobi space.   In some sense, \emph{a mobi space is to a mobi algebra in the same way as a module over a ring is to the ring of its scalars}.
For the moment we concentrate our attention on the unit interval as an abstract range of scalars. Here the slogan is \emph{the unit interval is to a mobi algebra in the same way as the real line is to a unitary ring}.  

The idea of abstracting the unit interval is not new and has been considered extensively in the literature (see e.g. \cite{Escardo} and the references therein). After some attempts and different experiments \cite{GTLab59}, it now appears to us that the structure which we are calling mobi algebra is a suitable abstraction for the unit interval capturing the main features we are interested in. It consists of a set $A$ equipped with a  ternary operation $p$ together with three constants $0,\muu,1\in A$, and satisfying the eight axioms of Definition \ref{mobi_algebra}. 
%is a reasonable abstraction to the unit interval ---reasonable in the sense that it captures the main features in which we are interested in. 
One of the reasons that have convinced us about this structure's significance is its deep connection with unitary rings (see diagram (\ref{casinha}) at the end). In summary, every unitary ring in which the element $2$ is invertible determines a unique structure of mobi algebra. In addition, a mobi algebra in which the element $\muu$ is invertible (in a suitable sense) determines a unique unitary ring structure. 

As expected, the prototyping example of a mobi algebra is the unit interval of the real numbers with the three constants $0$,  $\frac{1}{2}$, $1$,
 and the ternary operation $p(x,y,z)=(1-y)x+yz$. 
More generally, if $R$ is a unitary ring in which the element $2$ is invertible then any subset $A\subseteq R$, containing $0$, $2^{-1}$, $1$ and closed under the formula $x+yz-yx$, gives rise to a mobi algebra structure on the set $A$.

This paper is organized as follows. In Section \ref{sec_mobi} we give the definition of a mobi algebra and present some basic properties. Examples are given in Section \ref{sec_examples} and in the Appendix. In Section \ref{sec_derived}, we observe that a mobi algebra gives rise to several other structures. Namely, two monoid structures $(A,\circ,0)$, $(A,\cdot,1)$, dual to each other, one midpoint algebra $(A,\oplus)$ and an involution $\overline{()}\colon{A\to A}$. These derived operations satisfy some axioms and form what we call an involutive medial monoid, IMM for short. 
The importance of this structure is highlighted in Section \ref{sec_IMM-ring} where its relation to unitary rings is shown. 
In Section \ref{sec_ring} we characterize those IMM that are obtained from a mobi algebra. Our main result is Theorem \ref{THM-mobi-ring}
stating that every unitary ring with $2^{-1}$ uniquely determines and is uniquely determined by a mobi algebra with $\muu^{-1}$.

%Every finite mobi determines and is determined by a ring structure (example 15). This fact suggests the comparison with rings developed in Section \ref{sec_ring}. Our main result states that a unitary ring  determines and is determined by a mobi if, and only if, its element $2$ is invertible (Theorem \ref{Main_theorem_2}). It also states that a mobi $(A,p,0,\muu,1)$ determines and is determined by a unitary ring if, and only if, there exists an element $2\in A$ such that $p(0,\muu,2)=1$ (Theorem \ref{Main_theorem_1}), in other words, the element $\muu$ is invertible in the monoid $(A,\cdot,1)$.

\section{Definition and basic properties}\label{sec_mobi}

A mobi algebra consists of three constants and one ternary operation. The ternary operation $p(x,y,z)$ is thought of as being the point, in a path linking $x$ to $z$, which lies at a location $y$ in-between. %In other words, in order to have some intuition, 
This image might provide some intuition and the reader is invited to keep in mind either one of the two formulas $p=(1-y)x+yz$ or $p=x+y(z-x)$. The motivating example is the unit internal $[0,1]$, with the three constants $0$, $\frac{1}{2}$, $1$ and  the formula $p(x,y,z)=x+yz-yx$. Note that we will consistently distinguish between the element $\muu$, which is only used as a symbol, and the real number $\frac{1}{2}\in [0,1]$.
%along a line linking $x$ and $z$.

%The axioms that regulate this structure are the result of a long period of maturation. Initially we were considering an algebraic structure with three operations: one ternary, one binary and another one unary. This work is documented in the report \cite{GTLab59}. The observation that all the other operations (including the derived ones presented in the following sections) were determined from the ternary operation by fixing some constants have led us to the present structure.

\begin{definition}\label{mobi_algebra}
A mobi algebra is a system $(A,p,0,\muu,1)$, in which $A$ is a set, $p$ is a ternary operation and $0$, $\muu$ and $1$ are elements of A, that satisfies the following axioms:
\begin{enumerate}[label={\bf (A\arabic*)}]
\item\label{alg_mu} $p(1,\muu,0)=\muu$
\item\label{alg_01} $p(0,a,1)=a$
\item\label{alg_idem} $p(a,b,a)=a$
\item\label{alg_0} $p(a,0,b)=a$
\item\label{alg_1} $p(a,1,b)=b$
\item\label{alg_cancel} $p(a,\muu,b)=p(a',\muu,b)\implies a=a'$
\item\label{alg_homo} $p(a,p(c_1,c_2,c_3),b)=p(p(a,c_1,b),c_2,p(a,c_3,b))$
\item\label{alg_medial}
$p(p(a_1,c,b_1),\muu,p(a_2,c,b_2))=p(p(a_1,\muu,a_2),c,p(b_1,\muu,b_2))$
\end{enumerate}
\end{definition}

%{In a sequel of this paper \cite{functors}, we will consider the category of mobi algebras in which the morphisms are the maps that preserve the ternary operation and the three constants. Explicitly, if $(A,p,1,\muu,0)$ and $(B,p,1,\muu,0)$ are two mobi algebras then a map $f\colon{A\to B}$ is a morphism if $f(0)=0$, $f(\muu)=\muu$, $f(1)=1$ and 
%\[f(p(x,y,z))=p(f(x),f(y),f(z)).\]
%A variation of this definition of morphism, obtained by considering the ternary operation as a collection of binary operations, will also be considered in a future work.

The ternary operation $p$ is not associative, not even partially. Nevertheless, it verifies several properties that involve an interaction of $p$ with itself, as exemplified in the next Proposition. %and \ref{properties_p_3} below.

\begin{proposition}\label{properties_p_1} Let $(A,p,0,\muu,1)$ be a mobi algebra. It follows that:
\begin{eqnarray}
\label{P11} p(a,p(0,c,d),b)&=& p(a,c,p(a,d,b)) \\
\label{P12} p(a,p(1,c,d),b)&=& p(b,c,p(a,d,b)) \\
\label{P13} p(a,p(c,d,0),b)&=& p(p(a,c,b),d,a) \\
\label{P14} p(a,p(c,d,1),b)&=& p(p(a,c,b),d,b) 
\end{eqnarray}
\end{proposition}
\begin{proof}
These properties are obtained directly from axiom\ref{alg_homo} and the use of \ref{alg_0} or \ref{alg_1}.
\end{proof}

Fixing some other elements in the previous relations, we get further properties. For example, from (\ref{P12}), we get the following result.
\begin{corollary}\label{properties_p_2} If $(A,p,1,\muu,0)$ is a mobi algebra, then:
\begin{eqnarray}
\label{P21} p(a,p(1,c,0),b)&=& p(b,c,a) \\
\label{P22} p(a,\muu,b)&=& p(b,\muu,a)\\
\label{P23} p(1,p(1,c,0),0)&=& c
\end{eqnarray}
\end{corollary}
\begin{proof}
These properties are an immediate consequence of (\ref{P12}) and, respectively, axioms\ref{alg_0}, \ref{alg_mu}, and \ref{alg_01}.
\end{proof}

%\begin{proposition}\label{properties_p_3} Let $(A,p,0,\muu,1)$ be a mobi algebra. It follows that:
%\begin{eqnarray}
%\label{P31} p(a,\muu,p(a,c,b))&=& p(a,c,p(a,\muu,b)) \\
%\label{P32} p(a,\muu,p(b,c,a))&=& p(p(a,\muu,b),c,a)
%\end{eqnarray}
%\end{proposition}
%\begin{proof}
%These properties are obtained directly from axiom\ref{alg_medial} and the use of \ref{alg_idem}.
%\end{proof}

We finish this section with three more  properties of a mobi algebra structure.

\begin{proposition}\label{properties_p_4} Let $(A,p,0,\muu,1)$ be a mobi algebra. It follows that:
\begin{eqnarray}
\label{P41} p(p(a_1,c,b_1),\muu,p(a_2,c,b_2))&=& p(p(a_2,c,b_1),\muu,p(a_1,c,b_2))\hspace{10pt}\\
\label{P42} p(p(a,c,b),\muu,p(b,c,a))&=& p(a,\muu,b)\\
\label{P43} p(p(a,p(1,c,0),b),\muu,p(a,c,d))&=&p(a,\muu,p(b,c,d)). 
\end{eqnarray}
\end{proposition}
\begin{proof}
Using \ref{alg_medial}, (\ref{P22}) and again \ref{alg_medial}, we get:
\begin{eqnarray*}
p(p(a_1,c,b_1),\muu,p(a_2,c,b_2))&=&p(p(a_1,\muu,a_2),c,p(b_1,\muu,b_2))\\
                                 &=&p(p(a_2,\muu,a_1),c,p(b_1,\muu,b_2))\\
                                 &=&p(p(a_2,c,b_1),\muu,p(a_1,c,b_2)).
\end{eqnarray*}
Property (\ref{P42}) is just a particular case of (\ref{P41}) if we consider \ref{alg_idem}.\newline
Using \ref{alg_homo},\ref{alg_0}, \ref{alg_1}, \ref{alg_medial}, (\ref{P22}) and \ref{alg_idem}, we get:
\begin{eqnarray*}
& &p(p(a,p(1,c,0),b),\muu,p(a,c,d))\\
&=&p(p(p(a,1,b),c,p(a,0,b)),\muu,p(a,c,d))\\
&=&p(p(b,c,a),\muu,p(a,c,d))\\
&=&p(p(b,\muu,a),c,p(a,\muu,d))\\
&=&p(p(a,\muu,b),c,p(a,\muu,d))\\
&=&p(p(a,c,a),\muu,p(b,c,d))\\
&=&p(a,\muu,p(b,c,d)).
\end{eqnarray*}
\end{proof}

Some of these basic properties will be used to derive the structure introduced in section \ref{sec_derived} (involutive medial monoid) as an intermediate step in the comparison with rings. For the moment we observe some examples.

\section{Examples}\label{sec_examples}

In this section we give examples of mobi algebras by presenting a set $A$, a ternary operation $p(x,y,z)\in A$, for all $x,y,z\in A$, and three constants $0,\muu,1$ in $A$.

Our prototyping example is clearly the unit interval. We observe that it can be defined with the usual ternary operation $p(a,b,c)=(1-b)a+b c$ with one half as the element $\muu$ but it can also be considered in other forms like those of Examples 2 to 5.
Note that Examples 1 and 2 are isomorphic  via the bijection $x\mapsto\frac{x}{2-x}$.
More in general, the bijection $x\mapsto\frac{x}{(\alpha-1)+(2-\alpha)x}$ induces an automorphism leading to $\muu=\frac{1}{\alpha}$, $\alpha\in]1,+\infty[$ (see Example 3). Example 4 is isomorphic, via the bijective correspondence $x\mapsto 2x-1$, to Example 1. Via $x\mapsto\frac{1}{x}$ we easily observe that Example 5 is isomorphic to Example 1.

All examples have a mobi algebra structure $(A, p, 0, \muu, 1)$, as given in Definition \ref{mobi_algebra}, where the symbols $0$, $\muu$ and $1$ are explicit.

\begin{enumerate}[label={Example \arabic*:}]
\item
 $(A,p,0,\frac{1}{2},1)$, with $A=[0,1]$ and $$p(a,b,c)=(1-b)a+b c$$ for all $a,b,c\in A$

\item
$(A,p,0,\frac{1}{3},1)$, with $A=[0,1]$ and $$p(a,b,c)=\frac{a-ab+ac+2bc+abc}{1+b+c+2ab-bc}$$ for all $a,b,c\in A$

\item For any value of $\alpha\in]1,+\infty[$ we have
$(A,p,0,\frac{1}{\alpha},1)$, with $A=[0,1]$ and $$p(a,b,c)=\frac{a-ab+(\alpha-2)ac+(\alpha-1)bc+(\alpha-2)^2abc}{1+(\alpha-2)(b+c-bc)+(\alpha-1)(\alpha-2)ab}$$ for all $a,b,c\in A$

\item
 $(A,p,-1,0,1)$, with $A=[-1,1]$ and $$p(a,b,c)=\frac{a (1-b)+c (1+b)}{2}$$ for all $a,b,c\in A$

\item
 $(A,p,+\infty,2,1)$, with $A=[1,+\infty]$ and $$p(a,b,c)=\frac{abc}{a-c+bc}$$ for all $a,b,c\in A$

\item Excluding the trivial case, where $0=\muu=1$, the smallest mobi algebra is the set of constants $A=\{0,\muu,1\}$ with the operation $p$ defined as in the following tables.

\hspace{-0.6 cm}
$
\begin{tabular}{c|ccc}
%\hline 
p(-,0,-) & 0 & \muu & 1 \\ 
\hline 
0  & 0 & 0 & 0 \\ 
%\hline 
\muu & \muu & \muu & \muu \\ 
%\hline 
1 & 1 & 1 & 1 \\ 
%\hline 
\end{tabular} \,
$
$
\begin{tabular}{c|ccc}
%\hline 
p(-,\muu,-) & 0 & \muu & 1 \\ 
\hline 
0  & 0 & 1 & \muu \\ 
%\hline 
\muu & 1 & \muu & 0 \\ 
%\hline 
1 & \muu & 0 & 1 \\ 
%\hline 
\end{tabular} \,
$
$
\begin{tabular}{c|ccc}
%\hline 
p(-,1,-) & 0 & \muu & 1 \\ 
\hline 
0  & 0 & \muu & 1 \\ 
%\hline 
\muu & 0 & \muu & 1 \\ 
%\hline 
1 & 0 & \muu & 1 \\ 
%\hline 
\end{tabular} 
$
Isomorphically, we can consider the correspondence $(0,\muu,1)$ with $(-1,0,1)$, $(0,\frac{1}{2},1)$, $(0,1,2)$ or $(0,2,1)$ obtaining thus some natural structures on the given sets.

\item Considering the correspondence $(0,\muu,1)\mapsto (0,2,1)$, the previous example may be generalized as follows. For every number $n\in \mathbb{N}$, we have $(A,p,0,n+1,1)$ with $A=\{0,1,\ldots, 2n\}$ and $p(a,b,c)=a-b a+bc\mod 2n+1$.

\item The dyadic numbers as a subset of the unit interval, with the same structure as in Example 1. Clearly, if $a,b,c$ are dyadic numbers then $p(a,b,c)$ is dyadic.

\item $(A,p,0,\frac{1}{2},1)$ with $A$ as $\mathbb{Q}$, $\mathbb{R}$ or $\mathbb{C}$, and $p(x,y,z)=(1-y)x+yz$.

\item An example of a mobi algebra in $\mathbb{R}^2$ is $(A,p,0,\muu,1)$ with $$A=\left\{(x,y)\in\mathbb{R}^2\colon \vert y\vert \leq x \leq 1-\vert y\vert\right\}$$ 
\begin{eqnarray*}
p(a,b,c)=(a_1-b_1 a_1-b_2 a_2+b_1 c_1+b_2 c_2,\\
          a_2-b_1 a_2-b_2 a_1+b_1 c_2+b_2 c_1)
\end{eqnarray*}
and the three constants $\muu=\left(\frac{1}{2},0\right)$; $1=(1,0)$; $0=(0,0)$.

\item The previous example may be generalized, for any $K \in \mathbb{R}$, by defining
\begin{eqnarray*}
p(a,b,c)=((1-b_1) a_1+b_1 c_1&+& K b_2 (c_2-a_2),\\
          (1-b_1) a_2+b_1 c_2&+& b_2 (c_1-a_1)).
\end{eqnarray*}
With $\muu=\left(\frac{1}{2},0\right)$, $1=(1,0)$ and $0=(0,0)$, $(A,p,0,\muu,1)$ is a mobi algebra on $A=\mathbb{R}^2$ and, for $K \in \mathbb{R}^+_0$, on 
$$A=\left\{(x,y)\in\mathbb{R}^2\colon \sqrt{K} \vert y\vert \leq x \leq 1-\sqrt{K} \vert y\vert\right\}.$$
This example is obtained by identifying the plane with the ring of 2 by 2 matrices of the form 
\[(x,y)\mapsto \left(\begin{array}{cc}
x & k_1\, y \\ 
k_2\, y & x
\end{array} \right),
\]
letting $K=k_1\, k_2$ and computing $p(a,b,c)$ as $(1-b)a+bc$.

\item Any unitary ring $R$ in which the element $2$ is invertible gives an example of a mobi algebra $(R,p,0,2^{-1},1)$ with $p(x,y,z)=(1-y)x+yz$ (Theorem \ref{THM-ring2mobi}).

\item Any subset $S$ of a unitary ring (in which  $2$ is an invertible element) that is closed under the formula $p(x,y,z)=(1-y)x+yz$ and contains the three constants $0$, $2^{-1}$ and $1$ gives a mobi algebra $(S,p,0,2^{-1},1)$.

\item Let $(A,+,\cdot,0,1)$ be a semiring such that $2$ is invertible (i.e. $\exists\ 2^{-1}: 2\cdot 2^{-1}=1=2^{-1}\cdot 2$) and it exists $B\subseteq A$ with the following properties:
\begin{enumerate}[label=\roman*)]
\item $1, 2^{-1}, 0\in B$;
\item $\forall a,b,c \in B, \exists\ 1!$ solution $p=p(a,b,c)\in B$ to the equation $p+b\cdot a=a+b\cdot c$, 
\end{enumerate}
then the system $(B,p,0,2^{-1},1)$ is a mobi Algebra.
The same arguments that are used in the case of rings (Theorem \ref{THM-ring2mobi}) are valid here by rearranging the terms in order to avoid negative terms.

\item Every finite mobi is uniquely determined by (and uniquely determines) a unitary ring in which $2$ is invertible.
Indeed, let $(A,p,0,\muu,1)$ be a finite mobi algebra and consider the function $h:A\to A$ such that $h(x)=p(0,\muu,x)$. By axiom \ref{alg_cancel}, $h$ is injective. Now, as $A$ is a finite set, $h$ is also surjective. So $h$ is a bijection and, in particular, it exists $h^{-1}(1)$ which is a solution to the equation $p(0,\muu,x)=1$. In other words, $h^{-1}(1)$ is $\muu^{-1}$ and so Theorem \ref{THM-mobi-ring} holds.
\end{enumerate}

Note that Example 7 above illustrates the fact that every finite mobi must have an odd number of elements.

\section{Derived operations and IMM algebras}\label{sec_derived}

A closer look to the propositions of Section \ref{sec_mobi} suggests that some properties of mobi algebras can be suitably expressed in terms of a unary operation~"$\overline{()}$" and binary operations~"$\cdot$", "$\circ$" and "$\oplus$" defined as follows.
\begin{definition}\label{binary_definition} Let $(A,p,0,\muu,1)$ be a mobi algebra. We define:
\begin{eqnarray}
\label{def_complementar}\overline{a}&=&p(1,a,0)\\
\label{def_product}a\cdot b&=&p(0,a,b)\\
\label{def_star}a\oplus b&=&p(a,\muu,b)\\
\label{def_oplus}a\circ b&=&p(a,b,1).
\end{eqnarray}
\end{definition}
In the first example of previous section, with $p(a,b,c)=(1-b)a+b c$ on $A=[0,1]$, these operations have the following form:
\begin{eqnarray*}
\overline{a}&=&1-a\\
a\cdot b&=&ab\\
a\oplus b&=&\frac{a+b}{2}\\
a\circ b&=&a+b-ab.\\
\end{eqnarray*}
In Example 11, the derived operations are:
\begin{eqnarray*}
\overline{(x,y)}&=&(1-x, -y)\\
(x,y)\cdot (x',y')&=&(x x'+K y y', x y'+y x')\\
(x,y)\oplus (x',y')&=&(\frac{x+x'}{2},\frac{y+y'}{2})\\
(x,y)\circ (x',y')&=&(x+x'-x' x-K\, y' y,y+y'-x' y-y' x).\\
\end{eqnarray*}

%In $A=[0,1]$, $p(a,b,c)=\frac{a+ac-ab+2bc+abc}{1+b+c+2ab-bc}$ implies
%\begin{description}
%\item $a\cdot b=\frac{2ab}{1+a+b-ab}$
%\item $a\circ b=\frac{a+b}{1+ab}$
%\item $a\oplus b=\frac{a+b+2ab}{a+b+2}$
%\item $\overline{a}=\frac{1-a}{1+3a}$.
%\end{description}

%In $A=[-1,1]$, $p(a,b,c)=\frac{a (1-b)+c (1+b)}{2}$ implies
%\begin{description}
%\item $a\cdot b=\frac{a+b+ab-1}{2}$
%\item $a\circ b=\frac{a+b-ab+1}{2}$
%\item $a\oplus b=\frac{a+b}{2}$
%\item $\overline{a}=-a$.
%\end{description}

%In $A=[1,+\infty]$, $p(a,b,c)=\frac{abc}{a-c+bc}$ implies
%\begin{description}
%\item $a\cdot b=ab$
%\item $a\circ b=\frac{ab}{a+b-1}$
%\item $a\oplus b=\frac{2ab}{a+b}$
%\item $\overline{a}=\frac{a}{a-1}$.
%\end{description}

In particular, complex multiplication is obtained as $\cdot$, if letting $A=\mathbb{R}^2\cong \mathbb{C}$ and $K=-1$.

From property (\ref{P21}) and (\ref{P23}), as well as axioms \ref{alg_1} and \ref{alg_homo}, we immediately find that:
\begin{eqnarray}
\label{PC0} \overline{\overline{a}}&=& a\\
\label{PC1} \overline{1}&=&0\\
\label{PC3} p(b,c,a)&=&p(a,\overline{c},b)\\
\label{PC5} \overline{p(a,c,b)}&=&p(\overline{a},c,\overline{b}).
\end{eqnarray}
These relations show, in particular, that $\cdot$ and $\circ$ are dual operations in the sense that:
\begin{eqnarray}
\label{complementary_product}\overline{a\cdot b}&=&\overline{b}\circ\overline{a}\\
\label{complementary_oplus}\overline{a\circ b}&=&\overline{b}\cdot\overline{a}.
\end{eqnarray}
This is why, we will leave out the operation $\circ$ in the rest of the article, except for noting that (\ref{P41}) gives the following relation between the three binary operations:
\begin{eqnarray}
(a\circ b)\oplus(b\cdot a)=b\oplus a.
\end{eqnarray}
It is also worth noting that, \ref{alg_01} implies that
\begin{eqnarray}
\label{PCmu} \muu &=&\overline{1}\oplus 1.
\end{eqnarray}

Using (\ref{P43}), we can relate the ternary operation $p$ of any mobi algebra with the derived operations through the relation:
\begin{eqnarray}
\label{p_to_ring} \muu\cdot p(a,b,c)&=&(\overline{b}\cdot a)\oplus (b\cdot c).
\end{eqnarray}
This property will be at the bottom line of Section \ref{sec_mobiring} for comparing a mobi algebra with rings.

Before that, we show, in Proposition \ref{mobi2IMM} below, that every mobi algebra gives rise to a new structure, that we call involutive medial monoid (IMM), presented in the following Definition . 
\begin{definition}\label{def_IMM}
An IMM algebra is a system $(B,\overline{()},\oplus, \cdot, 1)$, in which $B$ is a set, $\overline{()}$ is an unary operation, $\oplus$ and $\cdot$ are binary operations and $1$ is an element of $B$, that satisfies the following axioms:
\begin{enumerate}[label={\bf (B\arabic*)}]
\item\label{IMM_idem} $a\oplus a=a$
\item\label{IMM_commut} $a\oplus b=b\oplus a$
\item\label{IMM_medial} $(a\oplus b)\oplus(c\oplus d)=(a\oplus c)\oplus(b\oplus d)$
\item\label{IMM_assoc} $a\cdot(b\cdot c)=(a\cdot b)\cdot c$
\item\label{IMM_1}$a\cdot 1=a=1\cdot a$
\item\label{IMM_distrib}$a\cdot (b\oplus c)=(a\cdot b)\oplus(a\cdot c),\quad (a\oplus b)\cdot c=(a\cdot c)\oplus(b\cdot c)$
\item\label{IMM_invol} $\overline{\overline{a}}=a$
\item\label{IMM_complem} $\overline{a\oplus b}=\overline{a}\oplus\overline{b}$
\item\label{IMM_0}$a\cdot \overline{1}=\overline{1}=\overline{1}\cdot a$ 
\item\label{IMM_muu} $\overline{a}\oplus a=\overline{1}\oplus 1$
\end{enumerate}
\end{definition}
The name IMM is chosen to highlight the existence of an involution, the presence of the medial law \ref{IMM_medial}\cite{Jezek} satisfied by $\oplus$ and the fact that $(B, \cdot, 1)$ is a monoid.

\begin{proposition}\label{mobi2IMM}
If $(A,p,0,\muu,1)$ is a mobi algebra and if $\overline{()}$, $\oplus$, and~$\cdot$ are defined as in (\ref{def_complementar}), (\ref{def_product}) and (\ref{def_star}), then $(A,\overline{()},\oplus, \cdot, 1)$ is an IMM algebra in which $\overline{1}=0$ and $\overline{1}\oplus 1=\muu$.
\end{proposition}
\begin{proof}%%%%%%%%%%%%%%%%%%%%%%%%%%%%%%%%%%%%%%%%%%%%%%%%%%%%%%%%%%%%%%%%%%%%%%%%%%%%%%%%%%%%%%%%%%%%%%%%%%%%%%%%%%%%%%%%%%
All axioms of an IMM are easily proved using the axioms of a mobi algebra and the properties presented in Section \ref{sec_mobi}. Indeed: \ref{IMM_idem} is a particular case of \ref{alg_idem}; \ref{IMM_commut} is just a rewriting of (\ref{P22}); \ref{IMM_medial} is a consequence of \ref{alg_medial}; \ref{IMM_assoc} follows from \ref{alg_0} and \ref{alg_homo}, like (\ref{P11}). The first equality in \ref{IMM_1} is \ref{alg_01} and the second comes from \ref{alg_1}. Left-distributivity in \ref{IMM_distrib} is deduced from \ref{alg_idem} and \ref{alg_medial} and right-distributivity from \ref{alg_homo}, as follows:
\begin{eqnarray*}
a\cdot (b\oplus c) &=&p(0,a,p(b,\muu,c))=p(p(0,\muu,0),a,p(b,\muu,c))\\
                   &=&p(p(0,a,b),\muu,p(0,a,c))=(a\cdot b)\oplus (a\cdot c);\\
(a\oplus b)\cdot c&=&p(0,p(a,\muu,b),c)=p(p(0,a,c),\muu,p(0,b,c))\\
                   &=&(a\cdot c)\oplus(b\cdot c).
\end{eqnarray*}
\ref{IMM_invol} is just a rewriting of (\ref{P23}); \ref{IMM_complem} is a consequence of \ref{alg_homo}, while \ref{IMM_0} can be proved through \ref{alg_idem}, \ref{alg_0} and \ref{alg_1}:
\begin{eqnarray*}
a\cdot \overline{1} &=&p(0,a,p(1,1,0)=p(0,a,0)=0=\overline{1}\\
\overline{1}\cdot a&=&p(0,p(1,1,0),a)=p(0,0,a)=0=\overline{1}.
\end{eqnarray*}
Using \ref{alg_mu}, \ref{alg_01}, \ref{alg_idem} and \ref{alg_medial}, we can prove \ref{IMM_muu}:
\begin{eqnarray*}
\overline{a}\oplus a &=&p(p(1,a,0),\muu,a)=p(p(1,a,0),\muu,p(0,a,1))\\
                   &=&p(p(1,\muu,0),a,p(0,\muu,1))=p(\muu,a,\muu)=\muu.
\end{eqnarray*}
\end {proof}%%%%%%%%%%%%%%%%%%%%%%%%%%%%%%%%%%%%%%%%%%%%%%%%%%%%%%%%%%%%%%%%%%%%%%%%%%%%%%%%%%%%%%%%%%%%%%%%%%%%%%%%%%%%%%%%%%%

A finite example of an IMM algebra which is not obtained from a mobi algebra is $(A,\overline{()},\oplus, \cdot, 1)$ with $A=\{0,\muu,1\}$ and the operations $\oplus$ and $\cdot$ as in the following tables.
\begin{center}
\begin{tabular}{c|ccc}
%\hline 
$\oplus$ & 0 & \muu & 1 \\ 
\hline 
0  & 0 & \muu & \muu \\ 
%\hline 
\muu & \muu & \muu & \muu \\ 
%\hline 
1 & \muu & \muu & 1 \\ 
%\hline 
\end{tabular} $\qquad$
\begin{tabular}{c|ccc}
%\hline 
$\cdot$ & 0 & \muu & 1 \\ 
\hline 
0  & 0 & 0 & 0 \\ 
%\hline 
\muu & 0 & \muu & \muu \\ 
%\hline 
1 & 0 & \muu & 1 \\ 
%\hline 
\end{tabular}
\end{center}
It is clear that $\oplus$ is not cancellative.

We finish this section with some properties of an IMM algebra. 

\begin{proposition}\label{properties_IMM} Let $(B,\overline{()},\oplus, \cdot, 1)$ be a IMM algebra. It follows that:
\begin{eqnarray}
\label{IMMP1} a\oplus(b\oplus c)&=&(a\oplus b)\oplus(a\oplus c)\\
\label{IMMP2} \overline{\overline{1}\oplus 1}&=&\overline{1}\oplus 1 \\
\label{IMMP3} \overline{a}=a&\Rightarrow& a=\overline{1}\oplus 1 \\
\label{IMMP4} (\overline{1}\oplus 1)\cdot a&=&\overline{1}\oplus a\\
\label{IMMP5} (\overline{1}\oplus 1)\cdot a&=&a\cdot(\overline{1}\oplus 1).
\end{eqnarray}
\end{proposition}
\begin{proof}
(\ref{IMMP1}) is a direct consequence of \ref{IMM_idem} and \ref{IMM_medial}; (\ref{IMMP2}) of \ref{IMM_commut}, \ref{IMM_invol}  and \ref{IMM_complem};  and (\ref{IMMP3}) is a consequence of \ref{IMM_idem} and \ref{IMM_muu}. To prove (\ref{IMMP4}) and (\ref{IMMP5}), we use \ref{IMM_distrib}, \ref{IMM_1} and \ref{IMM_0}:
\begin{eqnarray*}
(\overline{1}\oplus 1)\cdot a=(\overline{1}\cdot a)\oplus(1\cdot a)=\overline{1}\oplus a\\
a \cdot (\overline{1}\oplus 1)=(a\cdot \overline{1})\oplus(a\cdot 1)=\overline{1}\oplus a.
\end{eqnarray*}
\end{proof}

These are the main properties that will be used in the following sections.

%%%%%%%%%%%%%%%%%%%%%%%%%%%%%%%%%%%%%%%%%%%%%%%%%%%%%%%%%%%%%%%%%%%%%%%%%%%%%%%%%%%%%%%%%%%%%%%%%%%%%%%%%%%%%%%%%%
\section{IMM algebras and unitary rings}\label{sec_IMM-ring}%%%%%%%% NEW SECTION 

We are now going to see how, in an IMM algebra, the operation $\oplus$, under the existence of an inverse (in the sense of $\cdot$) to the element $\overline{1}\oplus 1$, gives rise to the additive structure of a unitary ring with one half. Let's begin by recalling that, in a monoid $(A,\cdot,1)$, if an element admits an inverse, the inverse is unique. Indeed, suppose that $x\cdot a=1=a\cdot x$ and $a\cdot x'=1=x'\cdot a$, then
\begin{eqnarray*}
x=x\cdot 1=x\cdot (a\cdot x')=(x\cdot a)\cdot x'=1\cdot x'=x'.
\end{eqnarray*}
As usual, the inverse of $a$ is denoted $a^{-1}$.
%\begin{definition}\label{existence_of_2}
So, in an IMM algebra $(A,\overline{()},\oplus, \cdot, 1)$,
when the equation $(\overline{1}\oplus 1)\cdot x=1$ has a solution, its unique solution is denoted $(\overline{1}\oplus 1)^{-1}$.
%\end{definition}
The fact that $\overline{1}\oplus 1$ is central in the monoid $(A,\cdot,1)$, as expressed in (\ref{IMMP5}), implies that its inverse, when it exists, is also central:
\begin{eqnarray}
\label{prop2-1} x\cdot (\overline{1}\oplus 1)^{-1}=(\overline{1}\oplus 1)^{-1}\cdot x,\quad \forall x\in A.
\end{eqnarray}
Indeed, if $x\cdot a=a\cdot x$ then $x\cdot a^{-1}=a\cdot a^{-1}\cdot x\cdot a^{-1}=a^{-1}\cdot x\cdot a^{-1}\cdot a=a^{-1}\cdot x$.

As we will see, the following proposition is essential to find the symmetric elements in the induced ring.
\begin{proposition}\label{properties_of_2}
If $(A,\overline{()},\oplus, \cdot, 1)$ is an IMM algebra and if $\overline{1}\oplus 1$ admits an inverse, then:
\begin{eqnarray}
\label{prop2-2} \overline{1}\oplus (\overline{1}\oplus 1)^{-1}=1.
\end{eqnarray}
\end{proposition}
\begin{proof}%%%%%%%%%%%%%%%%%%%%%%%%%%%%%%%%%%%%%%%%%%%%%%%%%%%%%%%%%%%%%%%%%%
This Property follows directly from (\ref{IMMP4}). 
\end{proof}%%%%%%%%%%%%%%%%%%%%%%%%%%%%%%%%%%%%%%%%%%%%%%%%%%%%%%%%%%%%%%%%%%%%

Before presenting the main results of this section through the Theorems \ref{Main_theorem_1} and \ref{Main_theorem_2}, we enumerate the axioms of a ring in Definition \ref{ring} below so we can refer to them in the subsequent demonstrations.
\begin{definition}\label{ring}
A unitary ring is a system $(R,+,\cdot,0,1)$ that satisfies the following axioms:
\begin{enumerate}[label={\bf (R\arabic*)}]
\item\label{R1} $(a+b)+c=a+(b+c)$
\item\label{R2} $a+b=b+a$
\item\label{R3} $a+0=a$
\item\label{R4} $\forall a, \exists -a: -a+a=0$
\item\label{R5} $a\cdot(b\cdot c)=(a\cdot b)\cdot c$
\item\label{R6} $a\cdot 1=a=1\cdot a$
\item\label{R7} $a\cdot(b+c)=(a\cdot b)+(a\cdot c),\quad (a+b)\cdot c=(a\cdot c)+(b\cdot c)$
\end{enumerate}
\end{definition}
Immediate consequences of the axioms are the following properties:
\begin{eqnarray}
\label{absorvente}a\cdot 0=&0&=0\cdot a\\
\label{ring_cancel}a+b=a'+b&\Rightarrow&a=a'.
\end{eqnarray}

The main result of this section shows the connection between an IMM algebra and a ring. It claims that an IMM algebra $(A,\overline{()},\oplus, \cdot, 1)$ determines a unique structure of a ring on its underlying set, if, and only if, its element $(\overline{1}\oplus 1)$ is invertible, i.e. that $(\overline{1}\oplus 1)^{-1}$ exists.

\begin{theorem}\label{Main_theorem_1}
Let $(A,\overline{()},\oplus, \cdot, 1)$ be an IMM algebra. The following three affirmations are equivalent.
\begin{enumerate}
\item[$(i)$] The equation $(\overline{1}\oplus 1)\cdot x=1$ has a solution in A;
\item[$(ii)$] There is a unique unitary ring $(A,+,\cdot,\overline{1},1)$ such that:
\begin{eqnarray*}\label{formula_A2R}
a+b=(1+1)\cdot (a\oplus b)
\end{eqnarray*}
\item[$(iii)$] There is a unique unitary ring $(A,+,\cdot,\overline{1},1)$ such that:
\begin{eqnarray*}\label{formula_A2Rbis}
a\oplus b&=&(\overline{1}\oplus 1)\cdot (a+b).
\end{eqnarray*}
\end{enumerate}
\end{theorem}
\begin{proof}%%%%%%%%%%%%%%%%%%%%%%%%%%%%%%%%%%%%%%%%%%%%%%%%%%%%%%

In order to present this proof in a concise way, we will use the following notations:
\begin{eqnarray*}
\muu&=&\overline{1}\oplus 1\\
\two&=&(\overline{1}\oplus 1)^{-1}.
\end{eqnarray*}
To prove that $(i)$ implies $(ii)$, we observe that $\muu\cdot\two=1$ and define
\begin{eqnarray*}
a+b&=&\two\cdot (a\oplus b)\\
-a&=&\overline{\two}\cdot a.
\end{eqnarray*}
To prove that the structure $(A,+,\cdot,\overline{1},1)$ is a ring, we will prove axioms \ref{R1} to \ref{R7} above. For \ref{R1}, we begin with a particular case of the IMM axiom \ref{IMM_medial} and,  using also \ref{IMM_commut}, \ref{IMM_1}, \ref{IMM_distrib} and property (\ref{IMMP4}), we get:
\begin{eqnarray*}
& & (a\oplus b)\oplus(\overline{1}\oplus c)=(a\oplus \overline{1})\oplus(b\oplus c)\\
&\implies&(a\oplus b)\oplus(\muu \cdot c)=(\muu \cdot a)\oplus(b\oplus c)\\
&\implies&(\muu\cdot\two\cdot(a\oplus b))\oplus(\muu \cdot c)=(\muu \cdot a)\oplus(\muu\cdot\two\cdot(b\oplus c))\\
&\implies&\muu\cdot(\two\cdot(a\oplus b)\oplus c)=\muu\cdot (a\oplus(\two\cdot(b\oplus c)))\\
&\implies&\two\cdot\two\cdot\muu\cdot((a+b)\oplus c)=\two\cdot\two\cdot\muu\cdot(a\oplus(b+c))\\
&\implies&(a+b)+ c=a+(b+c).
\end{eqnarray*}
\ref{R2} is a direct consequence of \ref{IMM_commut}, and \ref{R3} follows from the statement $0=\overline{1}$ and the use of property (\ref{IMMP4}). To prove \ref{R4}, we use \ref{IMM_distrib} and \ref{IMM_0}, after remarking that (\ref{prop2-2}), together with \ref{IMM_complem}, implies $1\oplus\overline{\two}=\overline{1}$:
\begin{eqnarray*}
-a+a=\two\cdot ((\overline{\two}\cdot a)\oplus a)=\two\cdot (\overline{\two}\oplus 1)\cdot a=\two\cdot \overline{1}\cdot a=\overline{1}.
\end{eqnarray*}
\ref{R5} and \ref{R6} are guaranteed by \ref{IMM_assoc} and \ref{IMM_1}. Finally \ref{R7} is deduced from \ref{IMM_distrib} with the use of (\ref{prop2-1}). 

This proves existence of a ring induced by an IMM when $(\overline{1}\oplus 1)^{-1}$ exists. To prove uniqueness, we just need to show that in any ring $(A,+',\cdot,\overline{1},1)$ such that $a+'b=(1+'1)\cdot (a\oplus b)$, we have $(1+'1)=(\overline{1}\oplus 1)^{-1}$. Indeed:
\begin{eqnarray*}
a+'b=(1+'1)\cdot (a\oplus b)&\implies& \overline{1}+'1=(1+'1)\cdot (\overline{1}\oplus1)\\
                            &\implies& 1=(1+'1)\cdot(\overline{1}\oplus1)\\
														&\implies& 1+'1=(\overline{1}\oplus 1)^{-1}.
\end{eqnarray*}
This also shows directly that $(ii)$ implies $(iii)$ because the inverse of $1+1$ is $\overline{1}\oplus1$. Now, as $1+1$ exists in any ring and $1\oplus 1=1$, we also have that $(iii)$ implies $(i)$.
\end{proof}%%%%%%%%%%%%%%%%%%%%%%%%%%%%%%%%%%%%%%%%%%%%%%%%%%%%%%%%

The previous proposition creates the question of characterizing those rings that come from an IMM algebra. Theorem \ref{Main_theorem_2} tell us that a ring $(R,+,\cdot,0,1)$ is determined by an IMM algebra structure if and only if its element $(1+1)$ is invertible. 
\begin{theorem}\label{Main_theorem_2}
Let $(R,+,\cdot,0,1)$ be a unitary ring. The following three affirmations are equivalent.
\begin{enumerate}
\item[$(i)$] The element $1+1$ admits an inverse in R;
\item[$(ii)$] There is a unique IMM algebra $(R,\overline{()},\oplus, \cdot, 1)$ such that:
\begin{eqnarray*}\label{formula_R2A}
\overline{a}&=&1-a\\
a\oplus b&=&(\overline{1}\oplus 1)\cdot (a+b).
\end{eqnarray*}
\item[$(iii)$] There is a unique IMM algebra $(R,\overline{()},\oplus, \cdot, 1)$ such that:
\begin{eqnarray*}\label{formula_R2Abis}
\overline{a}&=&1-a\\
a+ b&=&(1+1)\cdot (a\oplus b).
\end{eqnarray*}
\end{enumerate}
\end{theorem}
\begin{proof}%%%%%%%%%%%%%%%%%%%%%%%%%%%%%%%%%%%%%%%%%%%%%%%%%%%%%%

To prove that $(i)$ implies $(ii)$, we define
\begin{eqnarray*}
a\oplus b=(1+1)^{-1}\cdot (a+b).
\end{eqnarray*}
To prove that the structure $(A,\overline{()},\oplus, \cdot, 1)$ is an IMM algebra, we will deduce axioms \ref{IMM_idem} to \ref{IMM_muu} of Definition \ref{def_IMM}. \ref{IMM_idem} is satisfied because $a+a=(1+1)\cdot a$. \ref{IMM_commut} is a consequence of the commutativity of $+$. The medial law \ref{IMM_medial} may be proved using the associativity of $+$ and the distributivity of $\cdot$ over $+$:
\begin{eqnarray*}
(a\oplus b)\oplus (c\oplus d)&=&(1+1)^{-1}\cdot ((1+1)^{-1}(a+b)+(1+1)^{-1} (c+d))\\
                             &=&(1+1)^{-1}\cdot (1+1)^{-1}\cdot ((a+b)+ (c+d))\\
                             &=&(1+1)^{-1}\cdot (1+1)^{-1}\cdot ((a+c)+ (b+d))\\
														 &=&(a\oplus c)\oplus (b\oplus d).
\end{eqnarray*}
\ref{IMM_assoc} and \ref{IMM_1} are guaranteed by \ref{R5} and \ref{R6} while \ref{IMM_distrib} is guaranteed by \ref{R7} because $(1+1)^{-1}$ commutes with all the elements of the ring. \ref{IMM_invol} is a consequence of $1-(1-a)=a$. A proof of \ref{IMM_complem} goes like this:
\begin{eqnarray*}
\overline{a\oplus b}&=&1-(1+1)^{-1}(a+b)\\
                    &=&(1+1)^{-1}(1+1)-(1+1)^{-1}(a+b)\\
										&=&(1+1)^{-1}((1+1)-(a+b))\\
										&=&(1+1)^{-1}((1-a)+(1-b))\\
										&=&\overline{a}\oplus\overline{b}.
\end{eqnarray*}
\ref{IMM_0} is obvious because $\overline{1}=0$. Finally, to prove \ref{IMM_muu}, we observe that
\begin{eqnarray*}
\overline{a}\oplus a=(1+1)^{-1}(1-a+a)=(1+1)^{-1}
\end{eqnarray*}
And consequently $\overline{1}\oplus 1=(1+1)^{-1}$.
This proves existence of an IMM induced by a ring when $(1+1)^{-1}$ exists. To prove uniqueness, we just need to show that in any IMM $(A,\overline{()},\oplus', \cdot, 1)$ such that $a\oplus' b=(\overline{1}\oplus' 1)\cdot (a+b)$, we have $(\overline{1}\oplus' 1)=(1+1)^{-1}$. Indeed:
\begin{eqnarray*}
a\oplus' b=(\overline{1}\oplus' 1)\cdot (a+b)&\implies& 1\oplus' 1=(\overline{1}\oplus' 1)\cdot (1+1)\\
                                             &\implies& 1=(\overline{1}\oplus' 1)\cdot (1+1)\\
														                 &\implies& \overline{1}\oplus' 1=(1+1)^{-1}.
\end{eqnarray*}
This also shows directly that $(ii)$ implies $(iii)$ because the inverse of $\overline{1}\oplus1$ is $1+1$. Now, as $\overline{1}\oplus1$ exists in any IMM algebra and $\overline{1}+1=1$, we also have that $(iii)$ implies $(i)$.

\end{proof}%%%%%%%%%%%%%%%%%%%%%%%%%%%%%%%%%%%%%%%%%%%%%%%%%%%%%%%%

\section{Mobi algebras and IMM algebras}\label{sec_ring}%%%%%%%% NEW SECTION 
We have seen the comparison between IMM algebras and rings. We have also seen that  a mobi algebra gives rise to an IMM algebra (Proposition \ref{mobi2IMM}).  It remains to answer the question on whether an IMM algebra is obtained from a mobi algebra.
Due to axiom \ref{alg_cancel}, the subalgebra $(A, \oplus)$, of an IMM algebra which is induced by a mobi algebra, is a midpoint algebra \cite{Escardo, ccm_magmas}.  In other words, the operation $\oplus$ is cancellative. We present in the Appendix two examples of a IMM algebra in which $\oplus$ is not cancellative (IMM 2 and IMM 3). Note that the existence of $(\overline{1}\oplus 1)^{-1}$ is sufficient to imply that $\oplus$ is cancellative but it is not necessary.

An IMM algebra in which the operation $\oplus$ is cancellative will be called an IMM* algebra.
Therefore, the question, that will be answered in Theorem \ref{Main_theorem_3} below, is to determine what extra conditions on an IMM* algebra are needed to certify that it comes from a mobi algebra.
On an IMM* algebra, some axioms of IMM algebras may be deduced from the others using the cancellation of $\oplus$, thus we decided to present it as an independent algebraic structure.

\begin{definition}\label{def_IMM*}
An IMM* algebra is a system $(C,\overline{()},\oplus, \cdot, 1)$, in which $C$ is a set, $\overline{()}$ is an unary operation, $\oplus$ and $\cdot$ are binary operations and $1$ is an element of $C$, that satisfies the following axioms:
\begin{enumerate}[label={\bf (C\arabic*)}]
\item\label{IMM*_idem} $a\oplus a=a$
\item\label{IMM*_commut} $a\oplus b=b\oplus a$
\item\label{IMM*_cancel} $a\oplus b=a'\oplus b\implies a=a'$
\item\label{IMM*_medial} $(a\oplus b)\oplus(c\oplus d)=(a\oplus c)\oplus(b\oplus d)$
\item\label{IMM*_assoc} $a\cdot(b\cdot c)=(a\cdot b)\cdot c$
\item\label{IMM*_1}$a\cdot 1=a=1\cdot a$
\item\label{IMM*_distrib}$a\cdot (b\oplus c)=(a\cdot b)\oplus(a\cdot c),\quad (a\oplus b)\cdot c=(a\cdot c)\oplus(b\cdot c)$
\item\label{IMM*_0}$a\cdot \overline{1}=\overline{1}=\overline{1}\cdot a$ 
\item\label{IMM*_muu} $\overline{a}\oplus a=\overline{1}\oplus 1$
\end{enumerate}
\end{definition}

We now observe that, in particular, every IMM* algebra is a IMM algebra.

\begin{proposition}\label{properties_IMM*} Let $(C,\overline{()},\oplus, \cdot, 1)$ be a IMM* algebra. It follows that:
\begin{eqnarray}
\label{IMMP*1} \overline{\overline{a}}&=&a\\
\label{IMMP*2} \overline{a\oplus b}&=&\overline{a}\oplus\overline{b}\\
\label{IMMP*3} b\oplus a=\overline{1}\oplus1&\implies& b=\overline{a} \\
\label{IMMP*4} \overline{1}\oplus x=(\overline{b}\cdot a)\oplus( b\cdot c)&\implies&
\overline{1}\oplus \overline{x}=(\overline{b}\cdot \overline{a})\oplus( b\cdot \overline{c}).
\end{eqnarray}
\end{proposition}
\begin{proof}
Using \ref{IMM*_commut} and \ref{IMM*_muu}, we have $\overline{\overline{a}}\oplus\overline{a}=\overline{1}\oplus 1$ and $a\oplus \overline{a}=\overline{1}\oplus 1$ which, by \ref{IMM*_cancel}, implies (\ref{IMMP*1}). Using \ref{IMM*_medial}, \ref{IMM*_muu} and \ref{IMM*_idem} we find that
\begin{eqnarray*}
(\overline{a}\oplus \overline{b})\oplus(a\oplus b)=(\overline{a}\oplus a)\oplus(\overline{b}\oplus b)= 
(\overline{1}\oplus 1)\oplus (\overline{1}\oplus 1)=\overline{1}\oplus 1
\end{eqnarray*}
which implies (\ref{IMMP*2}) by cancellation, because $\overline{(a\oplus b)}\oplus (a\oplus b)= \overline{1}\oplus 1$. (\ref{IMMP*3}) is again a consequence of \ref{IMM*_cancel} and \ref{IMM*_muu}. To prove (\ref{IMMP*4}), suppose that 
\begin{eqnarray*}
\overline{1}\oplus x&=&(\overline{b}\cdot a)\oplus(b\cdot c),\quad \textrm{and}\\
\overline{1}\oplus y&=&(\overline{b}\cdot \overline{a})\oplus(b\cdot \overline{c}).
\end{eqnarray*}
Then, we have:
\begin{eqnarray*}
\overline{1}\oplus (x\oplus y)&=&(\overline{1}\oplus \overline{1})\oplus (x\oplus y)\\
&=&(\overline{1}\oplus x)\oplus (\overline{1}\oplus y)\\
&=&((\overline{b}\cdot a)\oplus(b\cdot c))\oplus((\overline{b}\cdot \overline{a})\oplus(b\cdot \overline{c}))\\
&=&((\overline{b}\cdot a)\oplus(\overline{b}\cdot \overline{a}))\oplus((b\cdot c)\oplus(b\cdot \overline{c}))\\
&=&(\overline{b}\cdot (a\oplus\overline{a}))\oplus(b\cdot (c\oplus\overline{c}))\\
&=&(\overline{b}\cdot (1\oplus\overline{1}))\oplus(b\cdot (1\oplus\overline{1}))\\
&=&(\overline{b}\oplus b)\cdot (1\oplus\overline{1})\\
&=&(\overline{1}\oplus 1)\cdot (1\oplus\overline{1})\\
&=& \overline{1}\oplus(\overline{1}\oplus 1).
\end{eqnarray*}
So, from \ref{IMM*_cancel}, we conclude that $x\oplus y=(\overline{1}\oplus 1)$ which means, using (\ref{IMMP*3}), that $y=\overline{x}$.
\end{proof}

We can now state the main result of this section.
\begin{theorem}\label{Main_theorem_3}
Let $(A,\overline{()},\oplus, \cdot, 1)$ be an IMM* algebra. The following two affirmations are equivalent.
\begin{enumerate}
\item[$(i)$] For each $a, b, c \in A$, the equation 
\begin{eqnarray}\label{def_p}
\overline{1}\oplus x=(\overline{b}\cdot a)\oplus( b\cdot c) 
\end{eqnarray}
has a solution x in A;
\item[$(ii)$] There is a unique mobi algebra $(A,p,\overline{1},\overline{1}\oplus 1,1)$ such that:
\begin{eqnarray}\label{formula_IMM2Mobi}
\overline{1}\oplus p(a, b, c)=(\overline{b}\cdot a)\oplus( b\cdot c).
\end{eqnarray}
\end{enumerate}
\end{theorem}
\begin{proof}
$(ii)$ implies $(i)$ because, for each $a, b, c \in A$, $p(a,b,c)$ exists if $(ii)$ is true and is therefore the solution of the equation (\ref{def_p}). To prove that $(i)$ implies $(ii)$, we will deduce axioms \ref{alg_mu} to \ref{alg_medial} of Definition \ref{mobi_algebra}. This is facilitated by the fact that $\oplus$ is cancellative which means that if we find an element $d$ of $A$ such that $\overline{1}\oplus d=(\overline{b}\cdot a)\oplus( b\cdot c)$, we can conclude that $p(a,b,c)=d$. To prove \ref{alg_mu}, we observe that 
\begin{eqnarray*}
(\overline{(\overline{1}\oplus 1)}\cdot 1)\oplus ((\overline{1}\oplus 1)\cdot\overline{1})=(\overline{1}\oplus 1)\oplus\overline{1}= \overline{1}\oplus (\overline{1}\oplus 1)
\end{eqnarray*}
which means that $p(1,\overline{1}\oplus 1,\overline{1})=\overline{1}\oplus 1$. In a similar way, \ref{alg_01} to \ref{alg_1} are satisfied:
\begin{eqnarray*}
(\overline{a}\cdot \overline{1})\oplus (a\cdot 1)=\overline{1}\oplus a\implies p(\overline{1},a,1)=a\\
(\overline{b}\cdot a)\oplus (b\cdot a)=(\overline{b}\oplus b)\cdot a=\overline{1}\oplus a\implies p(a,b,a)=a\\
(\overline{\overline{1}}\cdot a)\oplus (\overline{1}\cdot b)=(1\cdot a)\oplus\overline{1}=\overline{1}\oplus a\implies p(a,\overline{1},b)=a\\
(\overline{1}\cdot a)\oplus (1\cdot b)=\overline{1}\oplus b\implies p(a,1,b)=b.
\end{eqnarray*}
\ref{alg_cancel} is guaranteed by \ref{IMM*_cancel}. To prove \ref{alg_homo}, we first observe that
\begin{eqnarray*}
\overline{1}\oplus p(c_1,c_2,c_3)&=&\overline{c_2}\cdot c_1\oplus c_2\cdot c_3\\
\overline{1}\oplus p(a,c_1,b)&=&\overline{c_1}\cdot a\oplus c_1\cdot b\\
\overline{1}\oplus p(a,c_3,b)&=&\overline{c_3}\cdot a\oplus c_3\cdot b
\end{eqnarray*}
and, using (\ref{IMMP*4}):
\begin{eqnarray*}
\overline{1}\oplus \overline{p(c_1,c_2,c_3)}&=&\overline{c_2}\cdot \overline{c_1}\oplus c_2\cdot \overline{c_3}.
\end{eqnarray*}
Then, we use these relations, as well as the axioms of an IMM*, to transform an obvious identity into \ref{alg_homo}:
\begin{eqnarray*}
&&(\overline{c_2}\cdot \overline{c_1}\cdot a\oplus c_2\cdot\overline{c_3}\cdot a)\oplus
(\overline{c_2}\cdot c_1\cdot b\oplus c_2\cdot c_3\cdot b)=\\
&&(\overline{c_2}\cdot \overline{c_1}\cdot a\oplus c_2\cdot\overline{c_3}\cdot a)\oplus
(\overline{c_2}\cdot c_1\cdot b\oplus c_2\cdot c_3\cdot b)\\
&\Leftrightarrow&
(\overline{c_2}\cdot \overline{c_1}\cdot a\oplus \overline{c_2}\cdot c_1\cdot b)\oplus
(c_2\cdot\overline{c_3}\cdot a\oplus c_2\cdot c_3\cdot b)=\\
&&(\overline{c_2}\cdot \overline{c_1}\cdot a\oplus c_2\cdot\overline{c_3}\cdot a)\oplus
(\overline{c_2}\cdot c_1\cdot b\oplus c_2\cdot c_3\cdot b)\\
&\Leftrightarrow&
\overline{c_2}\cdot (\overline{c_1}\cdot a\oplus c_1\cdot b)\oplus
c_2\cdot(\overline{c_3}\cdot a\oplus c_3\cdot b)=\\
&&(\overline{c_2}\cdot \overline{c_1}\oplus c_2\cdot\overline{c_3})\cdot a\oplus
(\overline{c_2}\cdot c_1\oplus c_2\cdot c_3)\cdot b\\
&\Leftrightarrow&
\overline{c_2}\cdot (\overline{1}\oplus p(a,c_1,b))\oplus
c_2\cdot(\overline{1}\oplus p(a,c_3,b))=\\
&&(\overline{1}\oplus \overline{p(c_1,c_2,c_3)})\cdot a\oplus
(\overline{1}\oplus p(c_1,c_2,c_3))\cdot b\\
&\Leftrightarrow&
(\overline{1}\oplus\overline{c_2}\cdot p(a,c_1,b))\oplus
(\overline{1}\oplus c_2\cdot p(a,c_3,b))=\\
&&(\overline{1}\oplus \overline{p(c_1,c_2,c_3)}\cdot a)\oplus
(\overline{1}\oplus p(c_1,c_2,c_3)\cdot b)\\
&\Leftrightarrow&
\overline{1}\oplus(\overline{c_2}\cdot p(a,c_1,b)\oplus c_2\cdot p(a,c_3,b))=\\
&&\overline{1}\oplus (\overline{p(c_1,c_2,c_3)}\cdot a\oplus p(c_1,c_2,c_3)\cdot b)\\
&\Leftrightarrow&
p(p(a,c_1,b),c_2,p(a,c_3,b))=p(a,p(c_1,c_2,c_3), b).
\end{eqnarray*}
The proof of \ref{alg_medial} is similar. To write the proof in a concise way, let's use the notation $\overline{1}\oplus 1=\muu$ and recall that (\ref{IMMP5}) reads $\muu\cdot a=a\cdot\muu$ for all $a$ and (\ref{IMMP2}) means $\overline{\muu}=\muu$.
\begin{eqnarray*}
&&(\overline{c}\cdot \muu \cdot a_1\oplus c\cdot \muu\cdot b_1)\oplus
(\overline{c}\cdot \muu \cdot a_2\oplus c\cdot \muu\cdot b_2)=\\
&&(\muu\cdot\overline{c}\cdot a_1  \oplus \muu\cdot c\cdot b_1)\oplus
(\muu\cdot\overline{c}\cdot a_2  \oplus \muu\cdot c\cdot b_2)\\
&\Leftrightarrow&
(\overline{c}\cdot \muu \cdot a_1\oplus \overline{c}\cdot \muu \cdot a_2)\oplus
(c\cdot \muu\cdot b_1\oplus c\cdot \muu\cdot b_2)=\\
&&(\muu\cdot\overline{c}\cdot a_1  \oplus \muu\cdot c\cdot b_1)\oplus
(\muu\cdot\overline{c}\cdot a_2  \oplus \muu\cdot c\cdot b_2)\\
&\Leftrightarrow&
\overline{c}\cdot (\muu \cdot a_1\oplus \muu \cdot a_2)\oplus
c\cdot (\muu\cdot b_1\oplus\muu\cdot b_2)=\\
&&\muu\cdot(\overline{c}\cdot a_1  \oplus c\cdot b_1)\oplus
\muu\cdot(\overline{c}\cdot a_2  \oplus c\cdot b_2)\\
&\Leftrightarrow&
\overline{c}\cdot (\overline{1}\oplus p(a_1,\muu,a_2))\oplus
c\cdot (\overline{1}\oplus p(b_1,\muu, b_2))=\\
&&\muu\cdot(\overline{1}\oplus p(a_1,c,b_1))\oplus
\muu\cdot(\overline{1}\oplus p(a_2,c,b_2))\\
&\Leftrightarrow&
\overline{1}\oplus (\overline{c}\cdot p(a_1,\muu,a_2)\oplus
c\cdot p(b_1,\muu, b_2))=\\
&&\overline{1}\oplus(\muu\cdot p(a_1,c,b_1)\oplus
\muu\cdot p(a_2,c,b_2))\\
&\Leftrightarrow&
p(p(a_1,\muu,a_2),c,p(b_1,\muu, b_2))=
p(p(a_1,c,b_1),\muu,p(a_2,c,b_2)).
\end{eqnarray*}
\end{proof}

An IMM algebra in which $\overline{1}\oplus 1$ is invertible is an IMM* algebra. Moreover, $x=(\overline{1}\oplus 1)^{-1}\cdot((\overline{b}\cdot a)\oplus( b\cdot c))$ is a solution to (\ref{def_p}). Hence it gives rise to a mobi algebra.
\begin{corollary}\label{corollary_theorem_3}
Let $(A,\overline{()},\oplus, \cdot, 1)$ be an IMM algebra (or an IMM* algebra). If the element $\overline{1}\oplus 1$ is invertible, then there exists a unique mobi algebra $(A,p,\overline{1},\overline{1}\oplus 1,1)$ such that:
\begin{eqnarray}
p(a, b, c)=(\overline{1}\oplus 1)^{-1}\cdot((\overline{b}\cdot a)\oplus( b\cdot c)).
\end{eqnarray}
\end{corollary}
\begin{proof}
It is a consequence of Theorem \ref{Main_theorem_3} and property (\ref{IMMP4}).
\end{proof}

The previous results show the connection between mobi algebras and IMM algebras, or IMM* algebras. The case when the monoid part $(A,\cdot,1)$ of an IMM algebra is a commutative monoid has an interesting reflection on the axiom \ref{alg_medial}. Instead of having it restricted to the element $\muu$,
$$p(p(a_1,c,b_1),\muu,p(a_2,c,b_2))=p(p(a_1,\muu,a_2),c,p(b_1,\muu,b_2))$$
it holds for an arbitrary element as shown in the following proposition.
 
\begin{proposition}\label{commutativeA8}
Let $(A,\overline{()},\oplus, \cdot, 1)$ be an IMM* algebra satisfying condition (i) of Theorem \ref{Main_theorem_3} and suppose that $(A,p,\overline{1},\overline{1}\oplus 1,1)$ is its corresponding mobi algebra. The monoid $(A,\cdot,1)$ is a commutative monoid if and only if 
\begin{equation}\label{fullA8}
p(p(a_1,c,b_1),d,p(a_2,c,b_2))=p(p(a_1,d,a_2),c,p(b_1,d,b_2))
\end{equation}
for all $a_1,b_1,a_2,b_2,c,d\in A$.
\end{proposition}
\begin{proof}
If (\ref{fullA8}) is a property of the mobi algebra, then: 
\begin{eqnarray*}
a\cdot b &=&p(0,a,b)\\
         &=&p(p(0,b,0),a,p(0,b,1))\\
         &=&p(p(0,a,0),b,p(0,a,1))\\
         &=&p(0,b,a)\\
         &=&b\cdot a.
\end{eqnarray*}
Conversely and considering (\ref{formula_IMM2Mobi}), if $a\cdot b=b\cdot a$, for all $a,b\in A$ in the IMM*, then: 
\begin{eqnarray*}
&&((\overline{d}\cdot \overline{c}\cdot a_1)\oplus (\overline{d}\cdot c\cdot b_1))\oplus 
   ((d\cdot \overline{c}\cdot a_2)\oplus (d\cdot c \cdot b_2))=\\
&&((\overline{c}\cdot \overline{d}\cdot a_1)\oplus (c\cdot\overline{d}\cdot b_1))\oplus 
   ((\overline{c}\cdot d\cdot a_2)\oplus (c\cdot d \cdot b_2))\\
&\implies&
   ((\overline{d}\cdot \overline{c}\cdot a_1)\oplus (\overline{d}\cdot c\cdot b_1))\oplus 
   ((d\cdot \overline{c}\cdot a_2)\oplus (d\cdot c \cdot b_2))=\\
&&((\overline{c}\cdot \overline{d}\cdot a_1)\oplus (\overline{c}\cdot d\cdot a_2))\oplus
   ((c\cdot\overline{d}\cdot b_1)\oplus (c\cdot d \cdot b_2))\\
&\implies&
(\overline{d}\cdot ((\overline{c}\cdot a_1)\oplus (c\cdot b_1)))\oplus 
   (d\cdot ((\overline{c}\cdot a_2)\oplus (c \cdot b_2)))=\\
&&(\overline{c}\cdot ((\overline{d}\cdot a_1)\oplus (d\cdot a_2)))\oplus
   (c\cdot((\overline{d}\cdot b_1)\oplus (d \cdot b_2)))\\
&\implies&
(\overline{d}\cdot (\overline{1}\oplus p(a_1,c,b_1)))\oplus 
   (d\cdot (\overline{1}\oplus p(a_2,c,b_2)))=\\
&&(\overline{c}\cdot (\overline{1}\oplus p(a_1,d,a_2)))\oplus
   (c\cdot(\overline{1}\oplus p(b_1,d,b_2)))\\
&\implies&
(\overline{1}\oplus(\overline{d}\cdot p(a_1,c,b_1)))\oplus 
   (\overline{1}\oplus(d\cdot p(a_2,c,b_2)))=\\
&&(\overline{1}\oplus(\overline{c}\cdot p(a_1,d,a_2)))\oplus
   (\overline{1}\oplus(c\cdot p(b_1,d,b_2)))\\
&\implies&
\overline{1}\oplus((\overline{d}\cdot p(a_1,c,b_1))\oplus(d\cdot p(a_2,c,b_2)))=\\
&&\overline{1}\oplus((\overline{c}\cdot p(a_1,d,a_2))\oplus(c\cdot p(b_1,d,b_2)))\\
&\implies&
\overline{1}\oplus(\overline{1}\oplus p(p(a_1,c,b_1),d,p(a_2,c,b_2)))=\\
&&\overline{1}\oplus(\overline{1}\oplus p(p(a_1,d,a_2),c,p(b_1,d,b_2)))\\
&\implies&
p(p(a_1,c,b_1),d,p(a_2,c,b_2))=p(p(a_1,d,a_2),c,p(b_1,d,b_2)).
\end{eqnarray*}
\end{proof}

\section{Mobi algebras and unitary rings with one half}\label{sec_mobiring}%%%%%%%% NEW SECTION

We have seen the passage from mobi algebras to IMM(*) algebras and back (Proposition \ref{mobi2IMM} and Theorem \ref{Main_theorem_3}), as well as the passage from IMM algebras to unitary rings and back (Theorem \ref{Main_theorem_1} and Theorem \ref{Main_theorem_2}). Here we make explicit the fact that there is a straightforward connection between mobi algebras in which $\muu$ is invertible and unitary rings in which 2 is invertible.

\begin{theorem}\label{THM-mobi2ring}
Let  $(A,p,0,\muu,1)$ be a mobi algebra on the set $A$ and let $2\in A$ be a distinguished element on that set.\newline The structure $(A,+,\cdot,0,1)$, with $a\cdot b=p(0,a,b)$ and $a+b=2\cdot p(a,\muu,b)$, is a unitary ring if and only if $2$ is the inverse of $\muu$.
\end{theorem}
\begin{proof}
If $(A,+,\cdot,0,1)$ is a unitary ring, with $a+b=2\cdot p(a,\muu,b)$, then, in particular, $0+1=2\cdot p(0,\muu,1)$ which implies 
$1=2\cdot \muu$ proving that $2$ is the inverse of $\muu$. Conversely, we begin by using Proposition \ref{mobi2IMM} to obtain, from the mobi, an IMM structure $(A,\overline{()},\oplus, \cdot, 1)$ in which $\muu=\overline{1}\oplus 1$, $a\cdot b=p(0,a,b)$ and $a\oplus b=p(a,\muu,b)$. Within this structure, if $2$ is the inverse of $\muu$, we have, by Theorem \ref{Main_theorem_1}, that $a\oplus b=\muu\cdot (a+b)$ which is equivalent to $a+b=2\cdot p(a,\muu,b)$. 
\end{proof}

\begin{theorem}\label{THM-ring2mobi}
Let  $(A,+,\cdot,0,1)$ be a unitary ring with a distinguished element $\muu\in A$.\newline The structure $(A,p,0,\muu,1)$, 
with $p(a,b,c)=a+bc-ba$, is a mobi algebra if and only if $\muu$ is the inverse of $1+1$.
\end{theorem}
\begin{proof}
If $(A,p,0,\muu,1)$ is a mobi algebra, with $p(a,b,c)=a+bc-ba$, then axiom \ref{alg_mu}, $p(1,\muu,0)=\muu$, reads $1-\muu=\muu$, i.e. $(1+1)\cdot \muu=1$. On the other hand, if $\muu$ is the inverse of $1+1$, Theorem \ref{Main_theorem_2} gives us an IMM algebra $(A,\overline{()},\oplus, \cdot, 1)$ in which $\overline{b}=1-b$ and $a+b= (1+1)\cdot (a\oplus b)$. This implies, in particular, that 
$1=(1+1)\cdot(\overline{1}\oplus 1)$. Hence, through Corollary \ref{corollary_theorem_3}, we conclude that $(A,p,0,\muu,1)$ is a mobi algebra, with $p(a, b, c)=(\overline{1}\oplus 1)^{-1}\cdot(((1-b)\cdot a)\oplus( b\cdot c))=(a-ba)+bc$ . 
\end{proof}

\begin{theorem}\label{THM-mobi-ring} 
There is a bijective correspondence between unitary rings containing the element $2^{-1}$ and mobi algebras containing $2$.
\end{theorem}
\begin{proof}
Let $(A,+,\cdot,0,1)$ be a unitary ring such that $2^{-1}\in A$, with \mbox{$2=1+1$}. Then, by Theorem \ref{THM-ring2mobi}, the system $(A,p,0,2^{-1},1)$ where $p(a,b,c)=a+bc-ba$ is a mobi algebra. This mobi algebra contains $2$ (the inverse of $2^{-1}$) and, consequently, by Theorem \ref{THM-mobi2ring}, it determines a unitary ring $(A,+',\cdot',0,1)$. This ring is identical to the initial ring $(A,+,\cdot,0,1)$ because:
\begin{eqnarray*}
a\cdot'b&=&p(0,a,b)=0+a\cdot b-a\cdot 0=a\cdot b\\
a+'b&=&2\cdot'p(a,\frac{1}{2},b)=2\cdot(a+\frac{1}{2}\cdot b-\frac{1}{2}\cdot a)=a+b.
\end{eqnarray*}
Conversely, let $(A,p,0,\muu,1)$ be a mobi algebra such that  $p(0,\muu,2)=1$ with $2\in A$. Then, by Theorem \ref{THM-mobi2ring}, $(A,+,\cdot,0,1)$ with $a\cdot b=p(0,a,b)$ and $a+b=2\cdot p(a,\muu,b)$ is a unitary ring. This ring contains $\muu$ and, consequently, by Theorem \ref{THM-ring2mobi}, it determines a mobi algebra $(A,p',0,\muu,1)$. This mobi algebra is identical to the initial one. Indeed, by definition of $p'$, we have
\begin{eqnarray*}
p'(a,b,c)=(1-b)\cdot a + b\cdot c=2\cdot p((1-b)\cdot a,\muu,b\cdot c).
\end{eqnarray*}
Then, as shown in the proof of Proposition \ref{mobi2IMM}, we get $p(\overline{b},\muu,b)=\muu$. When 2 exists, this equality may be written as 
$\overline{b}+b=1$, i.e., $1-b=\overline{b}$. Therefore 
\begin{eqnarray*}
p'(a,b,c)=2\cdot p(\overline{b}\cdot a,\muu,b\cdot c)
\end{eqnarray*}
which, using property (\ref{p_to_ring}), implies that $p'(a,b,c)=p(a,b,c)$.
\end{proof}

The finite case is of particular interest because every finite mobi is such that its element $\muu$ is invertible, and hence it is uniquely determined by a unitary ring structure in which $2$ is invertible (see Example 15 in Section \ref{sec_examples}).

\section{conclusion}%%%%%%%%%%%%%%%%%%%%%%%%% NEW SECTION

We conclude with a schematic diagram relating the algebraic structures considered here and the results that relate them. We use an arrow labelled with the number of the Theorem, Proposition or Corollary where the result is proved on the direction indicated by the arrow. For example, the arrow labelled P.\ref{mobi2IMM}, with source Mobi and target IMM, simply means that the Proposition \ref{mobi2IMM} establishes an effective passage from the algebraic structure of a mobi algebra to the algebraic structure of an IMM algebra. Moreover, we use the name IMM** to designate an IMM* algebra in which condition (i) of Theorem \ref{Main_theorem_3} holds. This structure, by Corollary \ref{corollary_theorem_3}, is clearly in between IMM algebras, in which $\muu$ is invertible (that we are denoting by IMM$_2$), and IMM* algebras. Following the same line, we denote by Ring$_\muu$ the rings in which $2$ is invertible and by Mobi$_2$ the mobi algebras in which $\muu$ is invertible.
\begin{equation}\label{casinha}
\xymatrix{& \text{Ring}_\muu \ar@<.5ex>[ld]^{\textrm{T}.\ref{THM-ring2mobi}} \ar@<.5ex>[rd]^{\textrm{T}.\ref{Main_theorem_2}} \\
\text{Mobi}_2 \ar@<.5ex>[ur]^{\textrm{T}.\ref{THM-mobi2ring}} \ar@<.5ex>[rr]^{}\ar[d] & & \text{IMM}_2 \ar@<.5ex>[ul]^{\textrm{T}.\ref{Main_theorem_1}} \ar@<.5ex>[ll]^{\textrm{C}.\ref{corollary_theorem_3}} \ar[d]\\
\text{Mobi} \ar@<.5ex>[rr] \ar[d]^{\textrm{P}.\ref{mobi2IMM}} & & \text{IMM**} \ar@<.5ex>[ll]^{\textrm{T}.\ref{Main_theorem_3}} \ar[d]\\
\text{IMM}& & \text{IMM*} \ar[ll]}
\end{equation}

\medskip
We observe that in the proof of property P.4.1, the axiom \ref{alg_cancel} is not used which suggests that we could also consider a mobi without this axiom. Representing such a structure by $\text{Mobi}^\dagger$, we get the following additional diagram.
\begin{equation}\label{garagem}
\xymatrix{ \text{Mobi} \ar@<.0ex>[rd] \ar@<.0ex>[rr] & & \text{Mobi}^\dagger \ar@<.0ex>[dl]^{P.4.1}\\
&\text{IMM}&}
\end{equation}
%In the appendix, we present three examples of IMM algebras: the first one is a sub-algebra of a Mobi, the second one has a non-cancellative $\oplus$ operation and is a sub-algebra of a $\text{Mobi}^\dagger$, and the third one cannot be induced by mobi algebras.
Furthermore, the inclusion $\text{Mobi}\subset\text{Mobi}^{\dagger}$ is strict. Indeed, the example IMM 2 (see Appendix) is an IMM algebra that is obtained from a $\text{Mobi}^\dagger$ that is not a Mobi (axiom \ref{alg_cancel} is not satisfied). Moreover, there are IMM algebras which are not obtained from $\text{Mobi}^\dagger$ as it is shown by example IMM 3 in the Appendix.

Our last comment is that the connection between IMM algebras and rings can be lifted to the more general case of semi-rings. Indeed, if in the definition of IMM algebra we remove the unary operation $\overline{()}$ while keeping the existence of the element $\overline{1}$ such that $\overline{1}\cdot x=\overline{1}$ then, in Theorem \ref{Main_theorem_1}, we can replace rings by semi-rings.

\section*{Appendix}%%%%%%%%%%%%%%%%%%%%%%%%% NEW SECTION
In this appendix, we present some examples of finite IMM algebras with 5 elements. Let $A=\{\alpha, 0,\muu,1, \oalph\}$ be a set with 5 elements. In the three examples bellow, the unary operation $\overline{()}$ is defined by $\overline{\alpha}=\oalph$, $\overline{0}=1$, $\overline{1}=0$, $\overline{\muu}=\muu$, $\overline{\oalph}=\alpha$.
\begin{enumerate}[label={IMM \arabic*:}]
\item The system $(A,\overline{()},\oplus, \cdot, 1)$, with $\oplus$ and $\cdot$ defined as follows, is an IMM algebra. 
$$
\begin{array}{c|ccccc}
\oplus  & \alpha & 0      & \muu   & 1      & \oalph\\ 
\hline
\alpha  & \alpha & \oalph & 1      & 0      & \muu  \\ 
0       & \oalph & 0      & \alpha & \muu   & 1     \\ 
\muu    & 1      & \alpha & \muu   & \oalph & 0     \\ 
1       & 0      & \muu   & \oalph & 1      & \alpha\\  
\oalph  & \muu   & 1      & 0      & \alpha & \oalph\\ 
\end{array} \quad\quad\quad
\begin{array}{c|ccccc}
\cdot                     & \alpha & 0      & \muu   & 1      & \oalph\\
\hline 
\alpha                    & 1      & 0      & \oalph & \alpha & \muu  \\ 
0                         & 0      & 0      & 0      & 0      & 0     \\ 
\muu                      & \oalph & 0      & \alpha & \muu   & 1     \\ 
1                         & \alpha & 0      & \muu   & 1      & \oalph\\  
\oalph                    & \muu   & 0      & 1      & \oalph & \alpha\\ 
\end{array} 
$$
Defining $p(a,b,c)=\oalph\cdot((\overline{b}\cdot a )\oplus(b\cdot c))$, it can be checked that $(A,p,0,\muu,1)$ is a mobi algebra.

\item The system $(A,\overline{()},\oplus, \cdot, 1)$, with $\oplus$ and $\cdot$ defined as follows, is an IMM algebra. 
$$
\begin{array}{c|ccccc}
\oplus                    & \alpha & 0      & \muu   & 1      & \oalph\\ 
\hline
\alpha                    & \alpha & \muu   & \muu   & \alpha & \muu  \\ 
0                         & \muu   & 0      & \muu   & \muu   & \oalph\\ 
\muu                      & \muu   & \muu   & \muu   & \muu   & \muu  \\ 
1                         & \alpha & \muu   & \muu   & 1      & \muu  \\  
\oalph                    & \muu   & \oalph & \muu   & \muu   & \oalph\\ 
\end{array} \quad\quad\quad
\begin{array}{c|ccccc}
\cdot                     & \alpha & 0      & \muu   & 1      & \oalph\\
\hline 
\alpha                    & \alpha & 0      & \muu   & \alpha & \oalph\\ 
0                         & 0      & 0      & 0      & 0      & 0     \\ 
\muu                      & \muu   & 0      & \muu   & \muu   & \oalph\\ 
1                         & \alpha & 0      & \muu   & 1      & \oalph\\  
\oalph                    & \oalph & 0      & \oalph & \oalph & 0\\ 
\end{array} 
$$
It is obvious that the operation $\oplus$ is not cancellative. Nevetheless, the equation $\muu\cdot p=(\overline{b}\cdot a )\oplus(b\cdot c)$ can be solved for all $a, b, c \in A$. It can be shown that there are solutions $p$ of that equation for which $(A,p,0,\muu,1)$ is a $\text{Mobi}^\dagger$ (a mobi algebra without axiom \ref{alg_cancel}).

\item The system $(A,\overline{()},\oplus, \cdot, 1)$, with $\oplus$ and $\cdot$ defined as follows, is an IMM algebra.
$$
\begin{array}{c|ccccc}
\oplus                    & \alpha & 0      & \muu   & 1      & \oalph\\
\hline
\alpha                    & \alpha & \muu   & \muu   & \muu   & \muu  \\ 
0                         & \muu   & 0      & \muu   & \muu   & \muu  \\ 
\muu                      & \muu   & \muu   & \muu   & \muu   & \muu  \\ 
1                         & \muu   & \muu   & \muu   & 1      & \muu  \\  
\oalph                    & \muu   & \muu   & \muu   & \muu   & \oalph\\ 
\end{array} \quad\quad\quad
\begin{array}{c|ccccc}
\cdot                     & \alpha & 0      & \muu   & 1      & \oalph\\
\hline 
\alpha                    & \oalph & 0      & \muu   & \alpha & 1     \\ 
0                         & 0      & 0      & 0      & 0      & 0     \\ 
\muu                      & \muu   & 0      & \muu   & \muu   & \muu  \\ 
1                         & \alpha & 0      & \muu   & 1      & \oalph\\  
\oalph                    & 1      & 0      & \muu   & \oalph & \alpha\\ 
\end{array} 
$$
This IMM cannot be induced by a mobi algebra because the equation $\muu\cdot p=(\overline{b}\cdot a )\oplus(b\cdot c)$ doesn't have a solution $p$ for all $a,b,c \in A$, in contradiction with (\ref{p_to_ring}). Indeed, for example, the equation $\muu\cdot p=(\overline{\beta}\cdot\beta)\oplus(\beta\cdot\alpha)$ is equivalent to $\muu\cdot p=1$ which does not have a solution.
\end{enumerate}

\end{document}